\documentclass{article}
\usepackage{amscd,amsmath}
\usepackage[dvips]{graphicx}
\usepackage{hyperref}
\usepackage{array}
\usepackage[all]{xy}
\usepackage{amsfonts}
\usepackage{amsmath}

\newtheorem{theorem}{\bf Theorem}[section]
\newtheorem{proposition}{\bf Proposition}[section]
\newtheorem{lemma}{\bf Lemma}[section]

\newtheorem{definition}{\bf Definition}[section]
\newtheorem{remark}{\bf Remark}[section]

\newtheorem{notation}{\bf Notation}[section]
\newenvironment{proof}[1][Proof.]{\begin{trivlist} \item[\hskip \labelsep {\bfseries #1}]}{\end{trivlist}}
\title{Generic Simple Cocycles over Markov Maps}
\author{Mohammad Fanaee}
\date{}
\begin{document}

\maketitle

\begin{abstract}
Avila and Viana in [2] exhibit an explicit sufficient condition for the Lyapunov exponents of a linear cocycle over a Markov map to have multiplicity 1. Here, in terms of geometric perturbations, we prove that this sufficient criterion is generic in the space of all fiber bunched linear cocycles over Markov maps. Even more, the set of exceptional cocycles has infinite codimention, i.e. it is locally contained in finite unions of closed submanifolds with arbitrarily high codimension.\\
\end{abstract}

\setcounter{tocdepth}{1}

\tableofcontents

\section{Introduction} 

A linear cocycle over an invertible transformation $$f:N\rightarrow N$$  is a transformation $$F:N\times\mathbb C^d\rightarrow N\times\mathbb C^d$$ satisfying $f\circ \pi=\pi\circ F$ which acts by linear isomorphisms $A(x)$ on fibers. So, the cocycle has the form
$$F(x,v)=(f(x),A(x)v)$$
where
$$A:N\rightarrow\mathrm{GL}(d,\mathbb C).$$
Conversely, any $A:N\rightarrow\mathrm{GL}(d,\mathbb C)$ defines a linear cocycle over $f$. Note that $F^n(x,v)=(f^n(x),A^n(x)v)$, where
$$A^n(x)=A(f^{n-1}(x))~...~A(f(x))A(x),$$
$$A^{-n}(x)=(A^n(f^{-n}(x)))^{-1},$$
for any $n\geq 1$, and $A^0(x)=\mathrm{id}$.

Let $\mu$ be a probability measure invariant by $f$. Oseledets Theorem [10] states that there exist a Lyapunov splitting
 $$E_1(x)\oplus...\oplus E_k(x),~1\leq k=k(x)\leq d,$$
and Lyapunov exponents $$\lambda_1(x)>...>\lambda_k(x),$$
$$\lambda_i(x)=\lim_{|n|\rightarrow\infty}\frac{1}{n}\log\parallel A^n(x)v_i\parallel,~v_i\in E_i(x),~1\leq i\leq k,$$
at $\mu$-almost every point. Lyapunov exponents are invariant, uniquely defined at almost every $x$ and vary measurably with the base point $x$. Thus, Lyapunov exponents are constant when $\mu$ is ergodic and then $\{\lambda_1,...,\lambda_k\}$ is  called the Lyapunov spectrum of $A$.\\

One problem is to characterize when all exponents have multiplicity 1 meaning that the subspace $E_i$ of vectors $v_i\in\mathbb C^d$ that share the same value of $\lambda_i$ has dimension 1. 
Guivarc'h and Raugi [8], and Gol'dsheid and Margulis [7] have studied multiplicity 1 of Lyapunov exponents for independent random matrices.\\

There has been much recent progress on this problem, specially when the base dynamics is hyperbolic: Bonatti and Viana [6] obtained a general criterion for simplicity of Lyapunov spectrum for cocycles over shifts of finite type that satisfy the fiber bunching (domination) condition. This criterion has improved by Avila and Viana [2] for cocycles over any Markov structure, who used it to prove the Zorich-Kontsevich conjecture [3]. 

Indeed, [6] includes a proof of the genericity notion referring to a proof in [5] for genericity of non zero exponents.  This proof is based on sufficient criteria of Furstenberg for the existence of non-zero Lyapunov exponents for certain linear cocycles over hyperbolic transformations: non-existence of probability measures on the fibers invariant under the cocycle and under the holonomies of the stable and unstable foliations of the transformation.

Here, we prove the genericity of Avila and Viana simplicity criterion in [2] for linear cocycles over Markov maps, directly, by explicit geometric perturbations along periodic orbits and respective homoclinic orbits.

\subsection{Full countable shifts}
Suppose that $N=\mathbb N^{\mathbb Z}$, the full shift space with countably many symbols, and  $$f:N\rightarrow N$$ the shift map
$$f((x_n)_{n\in\mathbb Z})=(x_{n+1})_{n\in\mathbb Z}.$$
A cylinder of $N$ is any subset
$$[a_k,...;a_0;...,a_l]=\{x:~x_j=\iota_j,~j=k,...,l\}$$
of $N$.
We endowed $N$ with topology generated by cylinders. The local stable and local unstable sets of any $x\in N$ are defined as
$$W^s_{\mathrm{loc}}(x)=\{y:~x_n=y_n,~n\geq 0\}$$
and
$$W^u_{\mathrm{loc}}(x)=\{y:~x_n=y_n,~n<0\}.$$\\

Let $N_u=\mathbb N^{\{n\geq 0\}}$ and $N_s=\mathbb N^{\{n<0\}}$. The map
$$x\mapsto(x_s,x_u)$$
is a homeomorphism form $N$ onto $N_s\times N_u$ where $x_s=\pi_s(x)$ and $x_u=\pi_u(x)$, for natural projections $\pi_s:N\rightarrow N_s$ and $\pi_u:N\rightarrow N_u$. We also consider the maps $f_s:N_s\rightarrow N_s$ and $f_u:N_u\rightarrow N_u$ defined by
$$f_u\circ\pi_u=\pi_u\circ f,$$
$$f_s\circ\pi_s=\pi_s\circ f^{-1},$$
respectively.\\

Assume that $\mu_f$ is an ergodic probability measure for $f$. Let $\mu_s=(\pi_s)_*\mu_f$ and $\mu_u=(\pi_u)_*\mu_f$ be the images of $\mu_f$ under the natural projections. It is easy to see that $\mu_s$ and $\mu_u$ are ergodic probabilities for $f_s$ and $f_u$, respectively. Notice that $\mu_s$ and $\mu_u$ are positive on cylinders, by definition.\\

We say that $\mu_f$ has product structure if there exists a measurable density function $\omega:N\rightarrow(0,+\infty)$ such that
$$\mu_f=\omega(x)(\mu_s\times\mu_u).$$

\subsection{Fiber bunching condition}
Assume that $N$ is endowed with a metric d for which (i) $\mathrm d(f(y),f(z))\leq\theta(x)~\mathrm d(y,z)$, for all $y,z\in W^s_{\mathrm{loc}}(x)$, (ii) $\mathrm d(f^{-1}(y),f^{-1}(z))\leq\theta(x)~\mathrm d(y,z)$, for all $y,z\in W^u_{\mathrm{loc}}(x)$,
where $0<\theta(x)\leq\theta<1$, for all $x\in N$.\\

Let $A$ be an  $\eta$-H\"{o}lder continuous linear cocycle over $f$.

\begin{definition}
$A$ is fiber bunched if there exists some constant $\tau\in(0,1)$ such that
$$||A(x)||~||A(x)^{-1}||~\theta(x)^{\eta}<\tau,$$
for any $x\in N$.
\end{definition}

\begin{remark}
Fiber bunching is an open condition in $C^{r,\rho}(N,d,\mathbb C)$: if $A$ is a fiber bunched linear cocycle then any linear cocycle $B$ sufficiently $C^0$ close to $A$ is also fiber bunched, by definition.
\end{remark}

Our main result is\\
\textbf{Main Theorem.} Generic fiber bunched linear cocycles over the full shift map have simple Lyapunov spectrum. Even more, the set of exceptional cocycles has infinite codimention. \\

This implies the following more general cases\\
\textbf{Corollary 1.} Generic fiber bunched linear cocycles over any shift map have simple Lyapunov spectrum: the set of exceptional cocycles has infinite codimention.\\\\
\textbf{Corollary 2.} Generic fiber bunched linear cocycles over any Markov map have simple Lyapunov spectrum: the set of exceptional cocycles has infinite codimention.

\section{Holonomy Maps}
In this section, we study the existence, continuity and differentiability of holonomy maps as transformations over stable and unstable sets.

\begin{notation}
Set
$$\theta^n(x)=\theta(f^{n-1}(x))~...~\theta(x),~n\geq1.$$
for any function $\theta:N\rightarrow\mathbb C$.
\end{notation}

The next lemma is a key lemma for existence of stable and unstable holonomy maps.

\begin{lemma}
If $A$ is fiber bunched then there exists some constant $C>0$ such that
$$\parallel A^n(y)\parallel\parallel A^n(z)^{-1}\parallel\theta^n(x)^\eta\leq C\tau^n,$$
for any $y,z\in W^s_{\mathrm{loc}}(x)$, and all $n\geq 1$.
\end{lemma}

\begin{proof}
Submultiplicativity of norms implies that
$$\parallel A^n(y)\parallel\parallel A^n(z)^{-1}\parallel\leq\prod_{j=0}^{n-1}\parallel A(f^j(y))\parallel~\parallel A(f^j(z))^{-1}\parallel.$$
By regularity of cocycle $A$, there is $C_1>0$ such that
$$\parallel A(f^j(y))\parallel/\parallel A(f^j(x))\parallel\leq \exp(C_1\mathrm{d}(f^j(x),f^j(y))^\eta)\leq\exp(C_1\theta^j(x)^\eta\mathrm{d}(x,y)^\eta).$$

It is similar for $\parallel A(f^j(z))^{-1}\parallel/\parallel A(f^j(x))^{-1}\parallel$. So, the right hand side in lemma is bounded above by
$$\exp[C_1\sum_{j=0}^{n-1}\theta^j(x)^\eta(\mathrm{d}(x,y)^\eta+\mathrm{d}(x,z)^\eta)]\prod_{j=0}^{n-1}\parallel A(f^j(x))\parallel~\parallel A(f^j(x))^{-1}\parallel\theta^{n\eta}.$$

Since $\theta(x)<\theta<1$, the first factor is bounded by some uniform constant $C>0$, and fiber bunching implies that the second one is bounded by $\tau^n$.

The proof of Lemma 2.1 is now completed.
\end{proof}

\subsection{Existence of holohomies}
Set $H^n_{x,y}=A^n(y)^{-1}A^n(x)$.

\begin{definition}
A cocycle $A$ admits s-holonomy if
$$H^s_{x,y}=\lim_{n\rightarrow+\infty}H^n_{x,y}$$
exists for any pair of points $x,y$ in the same local stable set. u-holonomy is defined in a similar way, when $n\rightarrow-\infty$, for pairs of points in the same local unstable set.
\end{definition}

\begin{proposition}
If $A$ is fiber bunched then, for all $x$ and any $y\in W^s_{\mathrm{loc}}(x)$, s-holonomy $H^s_{x,y}$ exists, where
\\(a) $H^s_{x,y}=H^s_{z,y}.H^s_{x,z}$, for any $z\in W^s_{\mathrm{loc}}(x)$, and $H^s_{y,x}.H^s_{x,y}=\mathrm{id}$,
\\(b)$H^s_{f^j(x),f^j(y)}=A^j(y)\circ H^s_{x,y}\circ A^j(x)^{-1}$, for all $j\geq1$.
\end{proposition}

\begin{proof}
We have
$$\parallel H^{n+1}_{x,y}-H^n_{x,y}\parallel\leq\parallel A^n(x)^{-1}\parallel~\parallel A(f^n(x))^{-1}A(f^n(y))-\mathrm{id}\parallel~\parallel A^n(y)\parallel.$$

By continuity of $A$, there is $C_2>0$ such that the middle factor is bounded by
$$C_2\mathrm{d}(f^n(x),f^n(y))^\eta\leq C_2\theta^n(x)^\eta\mathrm{d}(x,y)^\eta,$$
and hence, by the last lemma
\begin{eqnarray}
\parallel H^{n+1}_{x,y}-H^n_{x,y}\parallel\leq CC_2\tau^n\mathrm{d}(x,y)^\eta.
\end{eqnarray}

As $\tau<1$, this implies that $H_n(x,y)$ is a Cauchy sequence, uniformly on $x,y$, and therefore, it is uniformly convergent. This proves the first part of proposition. (a) follows immediately from definition, and
\begin{eqnarray}
A^n(f^j(y))^{-1}A^n(f^j(x))=A^{j}(y)A^{n+j}(y)^{-1}A^{n+j}(x)A^j(x)^{-1}
\end{eqnarray}
proves (b).

The proof of Proposition 2.1 is now completed.
\end{proof}

\begin{remark}
As fiber bunching is an open condition, the constants in Lemma 2.1
and Proposition 2.1 may be taken uniform on some neighborhood $\mathcal U$ of $A$ when $A$ is fiber bunched.
\end{remark}

Note that the s-holonomies $H^s_{x,y}$ vary continuously on $(x,y)$ in the sense that the map
$$(x,y)\rightarrow H^s_{x,y}$$
is continuous on $W^s_n=\{(x,y):~f^n(y)\in W^s_{\mathrm{loc}}(x)\}$, for every $n\geq 0$. It is, in fact, a direct consequence of the uniform limit on (1) when $(x,y)\in W^s_0$, for instance. The general case $n>0$ follows immediately, by (b) of the last proposition .

Indeed, as the constants $C,\bar C$ may be taken uniformly on $\mathcal U$, the Cauchy estimate in (1) is also locally uniform on $A$. Therefore, one may consider this notion of dependence:
$$(A,x,y)\rightarrow H^s_{A,x,y}$$
is continuous on $\mathcal U\times W^s_n$, for all $n\geq 0$.

\subsection{Differentiability of holonomies}
We notice that the space of all H\'older continuous cocycles is a Banach space and so the tangent space at each point $A$ is naturally identified with this Banach space.

\begin{proposition}
If $A$ is fiber bunched then the map
$$B\mapsto H^s_{B,x,y}$$
is of class $C^1$ on $\mathcal U$, for any $y\in W^s_{\mathrm{loc}}(x)$, and
$$\partial_BH^s_{B,x,y}(\dot B)=\sum_{i=0}^{+\infty}B^i(y)^{-1}[H^s_{B,f^i(x),f^i(y)}B(f^i(x))^{-1}\dot B(f^i(x))-$$
$$\hspace{4cm}B(f^i(y))^{-1}\dot B(f^i(y)))H^s_{B,f^i(x),f^i(y)}]B^i(x).$$
\end{proposition}

\begin{proof}
First, we show that the expression of $\partial_BH^s_{B,x,p}$ is well-defined. Let $i\geq 0$.
\begin{eqnarray}
H^s_{B,f^i(x),f^i(y)}B(f^i(x))^{-1}\dot B(f^i(x))-B(f^i(y))^{-1}\dot B(f^i(y))H^s_{B,f^i(x),f^i(y)}
\end{eqnarray}
may be written as
$$(H^s_{B,f^i(x),f^i(y)}-\mathrm{Id})B(f^i(x))^{-1}\dot B(f^i(x))+B(f^i(y))^{-1}\dot B(f^i(y))(\mathrm{Id}-H^s_{B,f^i(x),f^i(y)})$$
$$+B(f^i(x))^{-1}\dot B(f^i(x))-B(f^i(y))^{-1}\dot B(f^i(y)).$$

By the last proposition, there is some uniform $\bar C>o$ such that the first term is bounded by
$$\bar C\mathrm{d}(f^i(x),f^i(y))^\eta\parallel B(f^i(x))^{-1}\parallel~\parallel\dot B(f^i(x))\parallel.$$
It is the same for second term. The third one is equal to
$$B(f^i(x))^{-1}[\dot B(f^i(x))-\dot B(f^i(y))]+[B(f^i(x))^{-1}-B(f^i(y))^{-1}]\dot B(f^i(y)),$$
and since $B^{-1}$ and $\dot B$ are H\"older continuous, using (2), it is bounded by
$$(||B^{-1}||_{0,0}\eta(\dot B)+\eta(B^{-1})||\dot B||_{0,0})\mathrm{d}(f^i(x),f^i(y))^\eta\leq\parallel B^{-1}\parallel_{0,\eta}\parallel\dot B\parallel_{0,\eta}\mathrm{d}(f^i(x),f^i(y))^\eta.$$

Hence (3) is bounded by
$$(2\bar C+1)C_3\parallel\dot B\parallel_{0,\eta}\mathrm{d}(f^i(x),f^i(y))^\eta\leq(2\bar C+1)C_3\parallel\dot B\parallel_{0,\eta}\theta^i(x)^\eta\mathrm{d}(x,y)^\eta$$
where $C_3=\sup\{\parallel B^{-1}\parallel_{0,\eta},~B\in\mathcal U\}$. So, the $i$th term in the expression of $\partial_Bh^s_{B,y,z}(\dot B)$ is bounded by
\begin{eqnarray}
C_4\parallel\dot B\parallel_{0,\eta}\theta ^i(x)^\eta\mathrm{d}(x,y)^\eta||B^i(p)^{-1}||||B^i(x)||\leq C_4\tau^{i}\mathrm{d}(x,y)^\eta\parallel\dot B\parallel_{0,\eta},
\end{eqnarray}
by fiber bunching hypothesis where $C_4=(2\bar C+1)C_3$. Therefore, as $\tau<1$, the series (3) does converge, uniformly.\\

Now, we should derivate $H^s_{B,x,y}$. By definition, $H^s_{B,x,y}$ is the uniform limit of $H^n_{B,x,y}=B^n(y)^{-1}B^n(x)$ when $n\rightarrow\infty$. Indeed, $H^n_{B,x,y}$ is a differentiable function of $B$ with derivative $\partial_BH^n_{B,x,y}(\dot B)$ equal to
$$\sum_{i=0}^{n-1}B^i(y)^{-1}[H^{n-i}_{B,f^i(x),f^i(y)}B(f^i(x))^{-1}\dot B(f^i(x))-B(f^i(y))^{-1}\dot B(f^i(y))H^{n-i}_{B,f^i(x),f^i(y)}]B^i(x),$$
for all $\dot B$ in tangent bundle and any $n\geq 1$.

It suffices to show that $\partial_BH^n_{B,x,y}$ converges uniformly to $\partial_BH^s_{B,x,y}$ as $n\rightarrow\infty$. By (2), for any $\tau_0\in(\tau,1)$,
$$\parallel H^s_{B,x,y}-H^n_{B,x,y}\parallel\leq CC_2\sum_{i=n}^{\infty}\tau^{i}\mathrm{d}(x,y)^\eta$$
which is bounded by
$$C_5\tau^{n}\mathrm{d}(x,y)^\eta\leq C_5\tau_0^n\mathrm{d}(x,y)^\eta,$$
for some uniform constant $C_5>0$. Then, for all $0\leq i\leq n$,
$$\parallel H^s_{B,f^i(x),f^i(y)}-H^{n-i}_{B,f^i(x),f^i(y)}\parallel\leq
C_5\tau_0^{(n-i)}\mathrm{d}(f^i(x),f^i(y))^\eta$$
bounded by
$$C_5\tau_0^{(n-i)}\theta^i(x)^\eta\mathrm{d}(x,y)^\eta.$$

It follows, by Lemma 2.1, that the difference between the $i$th terms in the expressions of $\partial_BH^s_{B,x,y}$ and $\partial_BH^n_{B,x,y}$ is bounded by
$$2C_3C_5\tau_0^{n-i}\theta^i(x)^\eta\mathrm{d}(x,y)^\eta||B^i(y)^{-1}||||B^i(x)||\leq 2C_3C_5\tau_0^{n-i}\tau^i\mathrm{d}(x,y)^\eta.$$
Combining with , $\parallel\partial_BH^s_{B,x,y}-\partial_BH^n_{B,x,y}\parallel$ is bounded by
$$\{2C_3C_5\tau_0^n\sum_{i=0}^{n-1}(\tau_0^{-1}\tau)^i+C_4\sum_{i=n}^{+\infty} \tau^{i}\}\mathrm{d}(x,y)^\eta\parallel\dot B\parallel_{0,\eta}.$$

Since $\tau,\tau_0$ and $(\tau_0^{-1}\tau)$ are strictly less that 1, therefore the series tends uniformly to $0$ as $n\rightarrow\infty$.

The proof of Proposition 2.2 is now completed.\\
\end{proof}

There exist dual expressions of last results for unstable holonomies, for points in $W^u_{\mathrm{loc}}(x)$. The dual of the last proposition is the following.

\begin{proposition}
If $A$ is fiber bunched then
$$\mathcal U\ni B\mapsto H^u_{B,x,y}$$
is of class $C^1$, and, for any $y\in W^u_{\mathrm{loc}}(x)$,
$$\partial_BH^u_{B,x,y}(\dot B)=-\sum_{i=1}^{+\infty}B^{-i}(y)^{-1}[H^u_{B,f^{-i}(x),f^{-i}(y)}B(f^{-i}(x))^{-1}\dot B(f^{-i}(x))$$
$$\hspace{4cm}-B(f^{-i}(y))^{-1}\dot B(f^{-i}(y))H^u_{B,f^{-i}(x),f^{-i}(y)}]B^{-i}(x).$$
\end{proposition}

\section{Perturbation Tools}
In this section, we begin to prove the Main Theorem, by perturbations along periodic orbits and homoclinic orbits, regarding to Avila and Viana simplicity criterion.\\

First, lets recall  Avila and Viana simplicity criterion. Consider the ergodic complete shift system $(f,\mu)$ where $\mu$ has product structure and let $A:N\rightarrow\mathrm{GL}(d,\mathbb C)$ be a linear cocycle over $f:n\rightarrow N$.

Suppose that $p$ is a periodic point of $f$, and $q$ a homoclinc point of $p$, i.e. $q\in W^u_{\mathrm{loc}}(p)$ and there is some multiple $m\geq 1$ of $\mathrm{per}(p)$ such that $f^m(q)\in W^s_{\mathrm{loc}}(p)$. We define the \textit{transition map}
$$\Psi_{A,p,q}:\mathbb C_p^d\rightarrow\mathbb C_p^d$$
by
$$\Psi_{A,p,q}=H^s_{f^m(q),p}A^m(q)H^u_{p,q}\in \mathrm{GL}(d,\mathbb C).$$

\begin{definition}
$A$ is pinching at $p$ if all eigenvalues of $A^{\mathrm{per}(p)}(p)$ have distinct absolute values. $A$ is twisting at $p,q$ if, for any pair of invariant subspaces $E_1,E_2$ of $A^{\mathrm{per}(p)}(p)$ with $\dim E_1+\dim E_2=d$,
$$\Psi_{A,p,q}(E_1)\cap E_2=\{\bf{0}\}.$$
A cocycle $A$ is simple if there exist some periodic point $p$ and some homoclinic point $q$ of $p$ such that $A$ is  pinching at $p$ and twisting at $p,q$.
\end{definition}

Then, Avila and Viana simplicity criterion is

\begin{theorem}
$\mathrm{[2]}$ If $A$ is simple then the Lyapunov spectrum of $A$ is simple.
\end{theorem}

\subsection{Perturbation along periodic orbits}
As we mentioned before, the tangent space at any H\"older continuous cocycle $A$ is identified naturally with the space of all  H\"older continuous maps on $N$ into the space of linear maps on $\mathbb C^d$. Indeed, we may give any tangent vector $\dot A$ as a H\"older continuous map which assigns to every point of $N$ a linear map on $\mathbb C^d$.

\begin{proposition}
Let $p$ be a periodic point of $f$ then the application
$$A\mapsto A^{\mathrm{per}(p)}(p)\in \mathrm{GL}(d,\mathbb C)$$
is a submersion at any H\'older continuous cocycle $A$, even restricted to tangent vectors supported in some neighborhood of $p$.
\end{proposition}
\begin{proof}
Assume that $p$ is a fixed point of $f$. It is easy to see that
$$\partial_AA(p)(\dot A)=\dot A(p).$$

Fix a neighborhood $U_p$ of $p$ such that $p$ is the unique point of its orbit in $U_p$. Let $\alpha:N\rightarrow[0,1]$ be a H\"older continuous function vanishing outside $U_p$, and $\alpha(p)=1$. For any $\mathcal A\in\mathrm{GL}(d,\mathbb C)$, define $\dot{\mathcal A}$ in the tangent bundle as
$$\dot{\mathcal A}(x)=\mathcal AA(p)^{-1}\alpha(x)A(x).$$

Note that $\dot{\mathcal A}$ is supported on $U_p$, and $\dot{\mathcal A}(p)=\mathcal A$. Hence $\partial_AA(p)(\dot{\mathcal A})=\mathcal A$, as we have claimed. It is similar when per$(p)>1$ where in this case
$$\partial_AA^{\mathrm{per}(p)}(p)(\dot A)=A(f^{\mathrm{per}(p)-1}(p))~...~\dot A(p)+~...~+\dot A(f^{\mathrm{per}(p)-1}(p))~...~A(p)$$
which, for tangent vectors supported on $U_p$, reduces to
$$A(f^{\mathrm{per}(p)-1}(p))~...~\dot A(p).$$
The proof of Proposition 3.1 is now completed.
\end{proof}

\subsection{Perturbation along homoclinic orbits}
Assume that $p$ is a periodic point of $f$ and $q$ some homoclinic point of $p$.
The derivative of $\Psi_{B,p,q}=H^s_{f^m(q),p}.B^m(q).H^u_{p,q}$ at a vector $\dot B$ is given by
\begin{eqnarray}
\begin{array}{l} \partial_BH^s_{f^m(q),p}(\dot B).B^m(q).H^u_{p,q}+\\
H^s_{f^m(q),p}\partial_BB^m(q)(\dot B)H^u_{p,q}+\\
H^s_{f^m(q),p}B^m(q)\partial_BH^u_{p,q}(\dot B)
\end{array}
\end{eqnarray}
where
$$\partial_BB^m(q)(\dot B)=B(f^{m-1}(q))~...~\dot B(q)+~...~+\dot B(f^{m-1}(q))~...~B(q),$$
by definition.

\begin{proposition}
The application
$$\mathcal U\ni B\mapsto \Psi_{B,p,q}$$
is a submersion, even restricted to tangent vectors $\dot B$ supported on a neighborhood of $q$, for any periodic point $p$ and each homoclinic pint $q$ of $p$.
\end{proposition}

\begin{proof}
Without loose of generality, we assume that $p$ is a fixed point of $f$, and $m=1$. Let $U_q$ be any neighborhood of $q$ which is disjoint from the orbit of $p$ and $\{f^j(q):~j\neq 0\}$. So, the expression in (5) reduces to
$$H^s_{f(q),p}\partial_BB(q)(\dot B)H^u_{p,q}=H^s_{f(q),p}\dot B(q)H^u_{p,q}.$$

Thus, $\partial_B\Psi_{B,p,q}$ is given by
$$\dot B\mapsto H^s_{f(q),p}\dot B(q)H^u_{p,q},$$
for any vector $\dot B$ supported on $U_q$. We claim that
$$\Phi(\dot B)=H^s_{f(q),p}\dot B(q)H^u_{p,q}$$
is surjective.

Let $\beta:N\rightarrow[0,1]$ be a H\"older continuous function vanishing outside $U_q$, where $\beta(q)=1$. For any $\mathcal B\in\mathrm{GL}(d,\mathbb C)$, define $\dot{\mathcal B}$ as
$$\dot{\mathcal B}(w)=(H^s_{B,f(q),p})^{-1}\mathcal BB(q)^{-1}\beta(w)B(w)(H^u_{B,p,q})^{-1}.$$

Note that $\dot{\mathcal B}(q)={H^s_{B,f(q),p}}^{-1}\mathcal B{H^u_{B,p,q}}^{-1}$, and so $\Phi(\dot{\mathcal B})=\mathcal B$, as we have claimed.

The proof of Proposition 3.2 is now completed.
\end{proof}

\subsection{The main perturbation}
Now, we consider the main perturbation including both periodic and homoclinic orbits.

\begin{proposition}
If $A$ is fiber bunched then the application
$$\Theta:\mathcal U\rightarrow\mathrm{GL}(d,\mathbb C)^2$$
$$\Theta(B)=(B(p),\Psi_{B,p,q}),$$
$B\in\mathcal U$, is a submersion, even restricted to the subspace of tangent vectors $\dot B$ supported on some neighborhoods of $p,q$.
\end{proposition}

\begin{proof}
Take $U_p$ such that $U_p\cap\mathrm{orb}(p)=\{p\}$, $U_p\cap\mathrm{orb}(q)=\emptyset$, and similarly $U_q$ so that $U_q\cap\mathrm{orb}(q)=\{q\}$, $U_q\cap\mathrm{orb}(p)=\emptyset$.

First note that, if $\dot B$ is a tangent vector supported on $U_p\cup U_q$, so, there exist two tangent vectors $\dot B_1$ supported on $U_p$, and $\dot B_2$ supported on $U_q$ such that $\dot B=\dot B_1+\dot B_2$. Indeed, we may assume that
\[\dot B_1(x)=\left\{\begin{array}{cl}
\dot B(x)&x\in U_P\\\textbf{0}&x\notin U_P
\end{array}\right.\]\\
and
\[\dot B_2(x)=\left\{\begin{array}{cl}
\dot B(x)&x\in U_q\\\textbf{0}&x\notin U_q.
\end{array}\right.\]\\
So
$$\partial_B\Theta(B)(\dot B)=\partial_B\Theta(B)(\dot B_1)+\partial_B\Theta(B)(\dot B_2)$$
which is equal to
$$(\partial_BB(p)(\dot B_1),\partial_B\Psi_{B,p,q}(\dot B_1))+(\partial_BB(\dot B_2),\partial_B\Psi_{B,p,q}(\dot B_2))=$$
$$(\dot B_1(p),\partial_B\Psi_{B,p,q}(\dot B_1))+(\textbf{0},\partial_B\Psi_{B,p,q}(\dot B_2))$$

By Proposition 3.1 and Proposition 3.2, for any $(\mathcal B_1,\mathcal B_2)\in\mathrm{GL}(d,\mathbb C)^2$, there exist tangent vectors $\dot{\mathcal B_1}$ supported on $U_p$, and then $\dot{\mathcal B_2}$ supported on $U_q$ such that
$$\partial_B\Psi_{B,p,q}(\dot{\mathcal B_2})=\mathcal B_2-\partial_B\Psi_{B,p,q}(\dot{\mathcal B}_1),$$
and therefore
$$\partial_B\Theta(B)(\dot{\mathcal B})=(\mathcal B_1,\mathcal B_2)$$
where $\dot{\mathcal B}=\dot{\mathcal B_1}+\dot{\mathcal B_2}$ is supported on $U_p\cup U_q$.

The proof of Proposition 3.3 is now completed.
\end{proof}

\section{Generic Simplicity}
In this section, we complete the proofs of Main Theorem and, Corollary 1 and Corollary 2.
 
\subsection{Pinching} Let $X$ be the subset of matrices $A\in \mathrm{GL}(d,\mathbb C)$ whose eigenvalues are not all distinct in norm. $X$ is closed and contained in a finite union of closed submanifolds of $\mathrm{GL}(d,\mathbb C)$ with codimention $\geq 1$.

Proposition 3.1 follows that the subset of cocycles $B\in\mathcal U$ for which $B^{\mathrm{per}(p)}(p)\in X$ is closed and contained in a finite union of closed submanifolds with codimention $\geq 1$.\\

For any $l\geq 1$, consider periodic points $p_1,...,p_l$. As a corollary of Proposition 3.1, the application
$$A\mapsto (A^{\mathrm{per}(p_1)}(p_1),...,A^{\mathrm{per}(p_l)}(p_l))\in \mathrm{GL}(d,\mathbb C)^l$$
is a submersion at any H\"older continuous cocycle $A$, even restricted to tangent vectors supported in some neighborhoods of $p_1,...,p_l$.

We imply that the subset of linear cocycles $A\in\mathcal U$ where $A^{\mathrm{per}(p_i)}(p_i)\in X$ is closed and contained in a finite union of closed submanifolds with codimention $\geq l$.

\subsection{Twisting}
The subset $Y$ of all pairs of matrices $(A,B)$ such that there exist $B$-invariant subspaces $E_1,E_2$ with $\dim E_1+\dim E_2=d$ where $A(E_1)\cap E_2\neq\{\textbf{0}\}$, is closed and contained in a finite union of closed submanifolds of positive codimention. Indeed, Fixing $E_1,E_2$, the application
$$\mathrm{GL}(d,\mathbb C)\ni A\mapsto A(E_1)\in\mathrm{Grass}(\dim E_1,d)$$
is a submersion. In the other hand,
$$\{A:~A(E_1)~\mathrm{do~not~intersect~transversally}~E_2 \}$$
is a submanifold with codimension$\geq 1$, since
$$\{E\in\mathrm{Grass}(\dim E_1,d):~E~\mathrm{do~not~intersect~transversally}~E_2 \}$$
is a submanifold of positive codimension. Now, for any fixed matrix $B$, the set $Y$ is contained in a finite number of submanifolds of positive codimension. So, $Y$ is contained in a finite number of submanifolds of positive codimension in $\mathrm{GL}(\mathbb C,d)^2$.\\

Therefore, by Proposition 3.3 the subset of cocycles $B\in\mathcal U$ so that $$(B^{\mathrm{per}(p)}(p),\Psi_{B,p,q})\in Y$$
is closed and contained in a finite union of closed submanifolds with positive codimension.\\

Given $l\geq 1$, if $q_1,...,q_l$, respectively as some homoclinic points of periodic points $p_1,...,p_l$, respectively, then the subset of cocycles $B\in\mathcal U$ for which
$$(B^{\mathrm{per}(p_i)}(p_i),\Psi_{B,p_i,q_i})\in Y$$
is closed and contained in a finite union of closed submanifolds with codimension $\geq l$. Since
the application
$$\hat\Theta:\mathcal U\rightarrow\mathrm{GL}(d,\mathbb C)^{2l}$$
$$\hat\Theta(B)=(B(p_1),...,B(p_l),\Psi_{B,p_1,q_1},...,\Psi_{B,p_l,q_l}),$$
is a submersion, even restricted to the subspace of tangent vectors $\dot B$ supported on some neighborhoods of $p_1,...,p_l,q_1,...,q_l$.

\subsection{Real valued cocycles}
All results in [2] and perturbation arguments of this section are valid for cocycles with values in $\mathrm{GL}(d,\mathbb R)$. But, in this case there is the possibility of existence of pairs of complex conjugate eigenvalues. Indeed, the subset of matrices whose eigenvalues are not all distinct in norm has non-empty interior in $\mathrm{GL}(d,\mathbb R)$.

The way to bypass this, is treated in [5] and [6]:\\
Excluding a codimension 1 subset of cocycles, one may assume that\\
(i) all the eigenvalues of $B^{\mathrm{per}(p)}(p)$ are real and have distinct norms, except for $c\geq 0$ pairs of complex conjugate eigenvalues,\\
(ii) $\Psi_{A,p,q}(E_1)\cap E_2=\{\bf{0}\}$, for any direct sums $E_1$ and $E_2$ of eigenspaces of $B^{\mathrm{per}(p)}(p)$ with $\dim E_1+\dim E_2\leq d$.

Avoiding another subset of positive codimention, we can choose a new periodic point $\hat p$ so that all the eigenvalues of $B^{\mathrm{per}(\hat p)}(\hat p)$ are real and distinct.

Now, in this way, for any $l\geq 1$, avoiding a codimension $l$ subset of cocycles, one may suppose that periodic points $\hat p_1,...,\hat p_l$ are defined.

The proof of Main Theorem is now completed.

\subsection{More general shifts}
Now, we recall ideas of Avila and Viana [2] to explain how cocycles over more general Markov maps can be reduced to the case of the full countable shift map.\\

We begin by subshifts of countable type. Let $\mathcal N$ be a finite or countable set and $T=(t_{ij})_{i,j\in\mathcal N}$  be the transition matrix meaning that every $t_{ij}$is either 0 ir 1. Define
$$N_T=\{(t_n)_{n\in\mathbb Z}\in\mathcal N^{\mathbb Z}:~(t_{{t_n},{t_{n+1}}})=1,~n\in\mathbb Z\},$$
and let $f_T:N_T\rightarrow N_T$ be the restriction to $N_T$ of the shift map on $\mathcal N^{\mathbb Z}$. By definition the cylinders $[.]$ of $\hat\Sigma$ are its intersections with the cylinders of the space $\mathcal N^{\mathcal Z}$. One-sided shift spaces $N_T^u$ and $N_T^s$, and cylinders $[.]^u$ and $[.]^s$ are defined analogously.\\

Let $\mu_T$ be a probability measure on $N_T$ invariant under $f_T$ and whose support contains some cylinder $[I]=[\tau_1,...,\tau_k-1]$ of $N_T$. The subset $N$ of points that return to $[I]$ infinitely many times in forward and backward time has full measure, by Poincar\'e recurrence. Let $r(x)\geq1$ be the first return time and 
$$f(x)=f_T^{r(x)}(x),~x\in N.$$

This return map $f:N\rightarrow N$ may be seen as a shift on $\mathbb N^{\mathbb Z}$: let $\{J(l):~l\in\mathbb N\}$ be an enumeration of the family of Cylinders of the form 

\begin{eqnarray}
[\tau_0;~\tau_1,...,\tau_r,...,\tau_{r+k=1}]
\end{eqnarray}

with $\tau_{r+i}=\tau_i$ for $i=0,1,...,k-1$ and $r\geq1$ minimum with this property, then
$$\mathbb N^{\mathbb Z}\rightarrow N,~(l_n)_{n\in\mathbb Z}\mapsto\bigcap_{n\in\mathbb Z}f^{-n}(J(l_n))$$
conjugates $f$ to the shift map. Also, if $\mu$ is the normalized restriction of $\mu_T$ to $N$ then $\mu$ is a probability measure invariant by $f$ and it is ergodic for $f$ if $\mu_T$ is ergodic for $f_T$. The measure $\mu$ is positive on cylinders since $[I]$ is contained in the support of $\mu_T$. It has product structure if $\mu_T$ has: any cylinder of $N_T$ is homeomorphic to a product of cylinders of $N_T^u$ and $N_T^s$.\\

Now, to each cocycle $A_T:N_T\rightarrow\mathrm{GL}(d,\mathbb C)$ defined over $f_T:N_T\rightarrow N_T$ we associate a linear cocycle $A:N\rightarrow\mathrm{GL}(d,\mathbb C)$ defined over $f:N\rightarrow N$ by
$$A(x)=A_T^{r(x)}(x).$$

\begin{proposition}
The Lyapunov exponents of $A$ have multiplicity 1 if and only if the Lyapunov exponents of $A_T$ have multiplicity 1.
\end{proposition}

\begin{proof}
Given any non zero vector $v$,
\begin{eqnarray}
\lim_{|n|\rightarrow\infty}\frac{1}{n}\log||A(x)v||=\lim_{|n|\rightarrow\infty}\frac{1}{n}\log||A_T^{S^n_ r(x)}(x)v||
\end{eqnarray}
where $$S^n_r(x)=\sum_{j=0}^{n-1}r(f^j(x)).$$
In the other hand, for $\mu$ almost every $x$, (7) is equal to
$$\lim_{|n|\rightarrow\infty}S^n_r(x)\lim_{|m|\rightarrow\infty}\frac{1}{m}\log||A_T^m(x)v||=\frac{1}{\mu_T([I])}\lim_{|m|\rightarrow\infty}\frac{1}{m}\log||A_T^m(x)v||,$$
since $\frac{1}{n}S_r^n(x)$ converges to $\int r~d\mu=\frac{1}{\mu([I])}$. So, the Lyapunov exponents of either cocycle have multiplicity 1 if and only if the neither have.

The proof is now completed.\\
\end{proof}

Notice that $A$ is H\"older continuous if $A_T$ is H\"older continuous, since the return time $r(x)$ is constant on each cylinder as in (6). So, we have the following, immediately.

\begin{proposition}
The map
\begin{eqnarray}
A_T\mapsto A\in\mathcal U
\end{eqnarray}
is a submersion.\\
\end{proposition}

The proof of the last proposition is done since the derivative of (8) is the identity map.

The proof of Corollary 1 is now completed.

\subsection{Markov maps}
More generally, let $g:M\rightarrow M$ be a transformation preserving a probability measure $\nu$ and assume there exists a return map $\hat f$ to some domain $D\subset\mathrm{supp}(\nu)$ which is a Markov map. This means that there exists a finite or countable partition $\{J(l):~l\in\mathbb N\}$ of $D$ such that (i) $g$ maps each $J(l)$ bijectively to the whole domain $D$ and (ii) for any sequence $(l_n)_n$ in $\mathbb N^{n\geq0}$,
$$\bigcap_{n\geq0}\hat f^{-n}(J(l_n))$$
consists of exactly one point. Then $\hat f$ may be seen as the shift map on $\mathbb N^{n\geq0}$.

The normalized restriction $\hat\mu$ of $\nu$ to the domain of $\hat f$ is a probability measure invariant by $\hat f$, and it is ergodic for $\hat f$ if $\nu$ is ergodic for $g$. As before, to any linear cocycle over $g$, one may associate a linear cocycle over $\hat f$, or its natural extension $f:N\rightarrow N$, for which the Lyapunov exponents of either have multiplicity 1 if an only if the Lyapunov exponents of the other have multiplicity 1. Similarly to Proposition 4.2 one may transfer the arbitrary codimension to the space of all fiber bunched linear cocycles over $g$.

This completes the proof of Corollary 2.\\\\
\textbf{Acknowledgments}. I would like to thanks M. Viana for all supports and advices, and J. Santamaria for useful conversations. This work is supported by a doctoral grant from CNPq - TWAS.

\textit{Mohammad Fanaee}

\textit{Instituto de Matem\' atica e Estat\'istica (IME)}

\textit{Universidade Federal Fluminense (UFF)}

\textit{Campus Valonguinho 24020-140}

\textit{Niter\'oi - RJ - Brazil}

\textit{Email: mf@id.uff.br}

\end{document}